\documentclass[fleqn,reqno,11pt,a4paper,final]{amsart}

\usepackage[a4paper,left=30mm,right=30mm,top=30mm,bottom=30mm,marginpar=20mm]{geometry}
\usepackage{mathtools}     
\mathtoolsset{showonlyrefs}
\usepackage{xcolor} 
\usepackage{amsmath}
\usepackage{amssymb}
\usepackage{amsthm}
\usepackage{amscd}
\usepackage[ansinew]{inputenc}
\usepackage{cite}
\usepackage{bbm}
\usepackage{color}
\usepackage[english=american]{csquotes}
\usepackage[final]{graphicx}
\usepackage[colorlinks=true, pdfstartview=FitV, linkcolor=black, citecolor=black, urlcolor=black]{hyperref}
\usepackage{calc}
\usepackage{mathptmx}
\usepackage{t1enc}

\numberwithin{equation}{section}
\newtheoremstyle{thmlemcorr}{10pt}{10pt}{\itshape}{}{\bfseries}{.}{10pt}{{\thmname{#1}\thmnumber{ #2}\thmnote{ (#3)}}}
\newtheoremstyle{thmlemcorr*}{10pt}{10pt}{\itshape}{}{\bfseries}{.}\newline{{\thmname{#1}\thmnumber{ #2}\thmnote{ (#3)}}}
\newtheoremstyle{remexample}{10pt}{10pt}{}{}{\bfseries}{.}{10pt}{{\thmname{#1}\thmnumber{ #2}\thmnote{ (#3)}}}
\newtheoremstyle{ass}{10pt}{10pt}{}{}{\bfseries}{.}{10pt}{{\thmname{#1}\thmnumber{ A#2}\thmnote{ (#3)}}}

\theoremstyle{thmlemcorr}
\newtheorem{theorem}{Theorem}
\numberwithin{theorem}{section}
\newtheorem{lemma}[theorem]{Lemma}
\newtheorem{corollary}[theorem]{Corollary}

\newtheorem{definition}[theorem]{Definition}

\theoremstyle{remexample}
\newtheorem{remark}[theorem]{Remark}

\theoremstyle{ass}

\newcommand{\Ccal}{\mathcal{C}}
\newcommand{\Rbb}{\mathbb{R}}
\newcommand{\Tbb}{\mathbb{T}}
\newcommand{\T}{\mathbb{T}}

\DeclareMathOperator{\dist}{dist}

\DeclareMathOperator{\diverg}{div}

\DeclareMathOperator{\supp}{supp}

\newcommand{\norm}[1]{\|#1\|}

\newcommand{\R}{\mathbb{R}}

\newcommand{\eps}{\epsilon}
\def\d{{\,\rm d}}

\newcommand{\vr}{\varrho}

\newcommand{\DC}{C^\infty_c((0,T)\times\T^d)}
\newcommand{\intT}[1]{\int_{\T^d}  #1  \ \dx}

\newcommand{\dx}{ dx}
\newcommand{\dt}{dt}

\renewcommand{\eps}{\varepsilon}
\renewcommand{\epsilon}{\varepsilon}
\renewcommand{\phi}{\varphi}

\begin{document}


\title[Energy Conservation for Compressible Euler]{Energy Conservation for the Compressible Euler and Navier-Stokes Equations with Vacuum}

\author{Ibrokhimbek Akramov}
\address{\textit{Ibrokhimbek Akramov:} Institute of Applied Mathematics, Leibniz University Hannover, Welfengarten~1, 30167 Hannover, Germany}
\email{akramov@ifam.uni-hannover.de}

\author{Tomasz D\k{e}biec}
\address{\textit{Tomasz D\k{e}biec:}  Institute of Applied Mathematics and Mechanics, University of Warsaw, Banacha 2, 02-097 Warszawa, Poland}
\email{t.debiec@mimuw.edu.pl}

\author{Jack Skipper}
\address{\textit{Jack Skipper:} Institute of Applied Mathematics, Leibniz University Hannover, Welfengarten~1, 30167 Hannover, Germany}
\email{skipper@ifam.uni-hannover.de}

\author{Emil Wiedemann}
\address{\textit{Emil Wiedemann:} Institute of Applied Mathematics, Leibniz University Hannover, Welfengarten~1, 30167 Hannover, Germany}
\email{wiedemann@ifam.uni-hannover.de}


\maketitle

\begin{abstract}
	We consider the compressible isentropic Euler equations on $\mathbb{T}^d\times [0,T]$ with a pressure law $p\in C^{1,\gamma-1}$, where $1\le \gamma <2$. This includes all physically relevant cases, e.g.\ the monoatomic gas. We investigate under what conditions on its regularity a weak solution conserves the energy.
	Previous results have crucially assumed that $p\in C^2$ in the range of the density, however, 
	  for realistic pressure laws this means that we must exclude the vacuum case. Here we improve these results by giving a number of sufficient conditions for the conservation of energy, even for solutions that may exhibit vacuum: Firstly, by assuming the velocity to be a divergence-measure field; secondly, imposing extra integrability on $1/\rho$ near a vacuum; thirdly, assuming $\rho$ to be quasi-nearly subharmonic near a vacuum; and finally, by assuming that $u$ and $\rho$ are H\"older continuous. We then extend these results to show global energy conservation for the domain $\Omega\times [0,T]$ where $\Omega$ is bounded with a $C^2$ boundary. We show that we can extend these results to the compressible Navier-Stokes equations, even with degenerate viscosity. 
%


\noindent\textsc{MSC (2010): 35Q31 (primary); 35Q10, 35L65, 76N10.}

\noindent\textsc{Keywords: Compressible Euler equations, Compressible Navier-Stokes equations, Vacuum, Onsager's conjecture, Energy conservation.}


\end{abstract}






\section{Introduction}

In recent years some substantial effort has been directed towards investigating the relation between energy (or, more generally, entropy) conservation and regularity of weak solutions to a given physical system of equations.

Onsager's conjecture states that a weak solution of the (three-dimensional) incompressible Euler system will conserve energy if it is H\"older regular with exponent greater than $1\slash3$. Otherwise it is possible for solutions to exist where anomalous dissipation of energy occurs. First results towards energy conservation for weak solutions are due to Eyink~\cite{Ey1994} and Constantin, E, Titi~\cite{constetiti}. The sharpest results in optimal Besov spaces are due to Cheskidov et al.~\cite{ChCoFrSh2008} and Fjordholm-Wiedemann~\cite{FjWi18}. Further, Bardos and Titi  \cite{bardos2018onsager}, Bardos-Titi-Wiedemann~\cite{BaTiWi2018}, and Drivas-Nguyen~\cite{DrNg} have extended these results to consider solutions on a bounded domain.

Investigating the possibility of analogous statements for other systems has become another lively direction of research. Sufficient regularity conditions for the energy to be conserved were studied for a number of models: compressible Euler~\cite{FGSW}, the full Euler system~\cite{DrivasEyink}, compressible Navier-Stokes~\cite{Yu}, or Euler-Korteweg~\cite{DGSGT}. A general class of first-order conservation laws was considered in~\cite{GMSG}, and in~\cite{BGSTW} on bounded domains.


Another direction of research was aimed towards the construction of $(1\slash3-\eps)-$H\"older continuous 
solutions to the incompressible Euler system that do \emph{not} conserve energy. With the application, and further refinements, of the method of convex integration this was achieved recently by Isett~\cite{Is2016} and by Buckmaster et al.\ \cite{buckmaster2017onsager}. Thus the famous conjecture of Lars Onsager for the incompressible Euler equations is fully resolved.

One of the major differences between incompressible and compressible fluid dynamics is the possible formation of \emph{vacuum} in the latter case. This means that the density of the fluid may become zero in some region. More precisely, consider the isentropic compressible Euler system  
\begin{equation}\label{eq:compressibleEuler}
\begin{aligned}
\partial_t(\rho u)+\diverg(\rho u\otimes u)+\nabla p(\rho)&=0,\\
\partial_t\rho+\diverg(\rho u)&=0,
\end{aligned}
\end{equation}
where $u$ denotes the velocity and $\rho$ the density of the fluid. We will specify the constitutive pressure law $p=p(\rho)$ later. It is classically known that conservation laws like~\eqref{eq:compressibleEuler} may develop singularities (shocks) in finite time, which prohibits the use of a smooth notion of solution. Rather, one works with solutions in the sense of distributions, which may be very rough. Suppose now the density were initially bounded away from zero, $\rho^0\geq c>0$. If the solution were smooth, then from the continuity equation $\partial_t\rho+\diverg(\rho u)=0$ it would easily follow (cf.~equation (7) in~\cite{DiLi89}) that $\rho$ remains bounded away from zero for all times. More precisely, this requires $u$ to have bounded divergence. However, there seems to be no way to guarantee that the velocity component of a weak solution of~\eqref{eq:compressibleEuler} has bounded divergence, and thus it can not be excluded that the solution spontaneously develops vacuum in finite time. In fact, to our knowledge it remains an outstanding open question whether this can actually occur for the compressible Euler or even Navier-Stokes equations.  

The formation of vacuum constitutes a degeneracy that, in many situations, vastly complicates the mathematical analysis of compressible models. For instance, the compressible Euler equations cease to be strictly hyperbolic in vacuum regions. In the context of the current contribution, densities close to zero invalidate the methods and results from previous works like~\cite{FGSW, GMSG, BGSTW}: There, it is a crucial assumption that the nonlinearities depend on the dependent variables in a twice continuously differentiable fashion, in order to treat them like a quadratic expression in the commutator estimates. For the system~\eqref{eq:compressibleEuler}, a typical and physically reasonable pressure law would be the polytropic one, i.e.\ $p(\rho)=\rho^\gamma$ with $\gamma>1$. The second derivative, however, is of order $\rho^{\gamma-2}$ and thus blows up at zero, at least if $\gamma<2$. But the regime $1<\gamma<2$ is precisely the relevant one (for instance, a monoatomic gas has $\gamma=5/3$).   

The starting point of our current work is the result of Feireisl, Gwiazda, \'Swierczewska-Gwiazda, Wiedemann for the compressible Euler system~\cite{FGSW}, which we quote below. It gives sufficient conditions, in terms of Besov regularity of a weak solution, for energy conservation, but only as long as vacuum is excluded. In the presence of vacuum, the relevant commutator estimate involving the pressure completely breaks down, and it turns out that substantially new techniques are required to fix this. To our knowledge, the only other result on energy conservation for non-$C^2$ nonlinearities is the one on active scalar equations~\cite{AkWi2018}, using however different techniques.

In the current article, we give a number of sufficient conditions to ensure energy conservation even after possible formation of vacuum.

First (Section~\ref{section:divergencevmeasure}), we consider the condition that the velocity be a so-called divergence-measure field; this notion is well-known in geometric measure theory and hyperbolic conservation laws, but it may seem a bit unmotivated to consider in the present situation. However, justification comes from the compressible Navier-Stokes system, whose a priori estimates ensure this condition. We extensively discuss the ramifications of our result with respect to the Navier-Stokes equations in Section~\ref{NSE}, where we also compare it to recent work of Cheng Yu~\cite{Yu}.

In Section~\ref{section:Holdercontinuity}, we identify as a sufficient condition for energy conservation an estimate for the quotient between the density and its mollification, see equation~\eqref{eq:mollifierratiocondition}. This, in itself, may seem rather artificial, and we go on to identify more natural conditions that will ensure~\eqref{eq:mollifierratiocondition} to hold. Arguably, our strongest result is Corollary~\ref{HolderCor}: Under the slightly stronger assumption of H\"older (instead of Besov) regularity, but with the expected exponents, we can show energy conservation \emph{no matter how the density behaves near vacuum}. It is surprising that this result is completely agnostic to the way that $\rho$ approaches zero. It crucially relies on a new measure-theoretic observation (Lemma~\ref{mollificationlem}) that may be of independent interest.

If one does want to assume only Besov regularity, then one needs to make further assumptions on the density near vacuum; we show that energy is conserved provided the density descends into vacuum sufficiently fast (Corollary~\ref{fastdescent}) or sufficiently slowly (Corollary~\ref{slowdescent}).

Finally, in Section~\ref{section:with boundary} we demonstrate how to extend our results, so far shown only under periodic boundary conditions, to the case of a bounded domain.

\subsection{The result of Feireisl et al.}
To formulate the local or global energy equality for~\eqref{eq:compressibleEuler} it is useful to define the so-called pressure potential by
\begin{equation*}
P(\rho)=\rho\int_1^\rho\frac{p(r)}{r^2}dr.
\end{equation*}
The following theorem was proven 
in~\cite[Theorem~4.1]{FGSW}.
\begin{theorem}\label{thm:compressibleonsager}
	Let $\rho$, $u$ be a solution of~\eqref{eq:compressibleEuler} in the sense of distributions. Assume 
	\begin{equation}
	u\in B_3^{\alpha,\infty}((0,T)\times\Tbb^d),\hspace{0.3cm}\rho, \rho u\in B_3^{\beta,\infty}((0,T)\times\Tbb^d),\hspace{0.3cm}
	0 \leq \underline{\rho} \leq \rho \leq \overline{\rho} \ \mbox{a.e. in } (0,T)\times\Tbb^d,
	\end{equation}
	for some constants $\underline{\vr}$, $\overline{\vr}$, and
	$0\leq\alpha,\beta\leq1$ such that
	\begin{equation}
	\beta > \max \left\{ 1 - 2 \alpha; \frac{1 - \alpha}{2} \right\}.
	\end{equation}
	Assume further that $p \in C^2[\underline{\vr}, \overline{\vr}]$, and, in addition
	\begin{equation}
	p'(0) = 0 \ \mbox{as soon as}\ \underline{\vr} = 0.
	\end{equation}
	Then the energy is locally conserved, i.e.
	\begin{equation}\label{eq:energyidentity}
	\partial_t\left(\frac{1}{2}\rho|u|^2+P(\rho)\right)+\diverg\left[\left(\frac{1}{2}\rho|u|^2+p(\rho)+P(\rho)\right)u\right]=0
	\end{equation}
	in the sense of distributions on $(0,T)\times\Tbb^d$.
\end{theorem}

Our aim in the current paper is to improve the above theorem by relaxing the $C^2$ assumption on the pressure. This will allow, for instance, to apply the theorem in the physically relevant case of the isentropic pressure law $p(\rho)=\kappa\rho^\gamma$ with the adiabatic coefficient $\gamma\in(1,2)$, without excluding vacuum.


\section{Preliminaries}\label{besov}

\subsection{Function spaces}

For $\Omega:=(0,T)\times\Tbb^d$ we recall the Besov spaces $B_p^{\alpha,\infty}(\Omega)$  
 which is the space of tempered distributions $w$ for which the norm
\begin{equation}\label{besovshift}
\norm{w}_{B_p^{\alpha,\infty}(\Omega)}:=\norm{w}_{L^p(\Omega)}+\sup_{\xi\in\Omega}\frac{\norm{w(\cdot+\xi)-w}_{L^p(\Omega\cap(\Omega-\xi))}}{|\xi|^\alpha}
\end{equation}
is finite. 
The above norm provides a control over shifts of the distribution $w$, making Besov spaces a convenient environment for our analysis, as it relies on convolutions with a mollifying kernel.

Let $\eta\in C_c^\infty\left(\Rbb^N\right)$ be a positive, radial function of integral $1$ with 
\begin{equation}
\eta(x)=\begin{cases} 1\quad \mbox{for} \quad
|x|\le \frac13,\\
0\quad\mbox{for}\quad |x|\ge 1,
\end{cases}
\end{equation}
and for $N=1+d$ set 
\begin{equation*}
\eta^\epsilon(x)=\frac{1}{\epsilon^{N}}\eta\left(\frac{x}{\epsilon}\right).
\end{equation*}
We define the notation $w^\epsilon:=\eta^\epsilon*w$. For any function $w$, $w^\epsilon$ is well-defined on $\Omega^\epsilon=\{x\in\Omega: \d(x,\partial\Omega)>\epsilon\}$.


It is then easy to check that the definition of the Besov spaces implies
\begin{equation}\label{besoveps}
\norm{w^\epsilon-w}_{L^p(\Omega^\eps)}\leq C\epsilon^\alpha\norm{w}_{B_p^{\alpha,\infty}(\Omega)}
\end{equation}
and
\begin{equation}\label{besovepsgradient}
\norm{\nabla w^\epsilon}_{L^p(\Omega^\eps)}\leq C\epsilon^{\alpha-1}\norm{w}_{B_p^{\alpha,\infty}(\Omega)}.
\end{equation}

By $\mathcal{M}(\Omega)$ we denote the space of signed Radon measures equipped with the total variation norm
\begin{equation}
\norm{\mu}_{TV}\coloneqq \int_{\Omega}\mathrm{d}|\mu|.
\end{equation}


%
%

\subsection{Derivation of the local energy equality}
The starting point in the proof of Theorem~\ref{thm:compressibleonsager}, as well as all our results, is to mollify the Euler equations, then derive the local energy equality for the regularized quantities, and finally estimate commutator errors generated by nonlinear terms. As this strategy is a common part in the proofs of our theorems, we devote this section to the said derivation, omitting the details of passing to the limit under the assumptions of Theorem~\ref{thm:compressibleonsager}. 


We begin by mollifying the momentum equation in time and space to obtain
\begin{equation}\label{eq:mollifiedmomentum}
\partial_t(\rho u)^\eps+\diverg(\rho u\otimes u)^\eps+\nabla p^\eps(\rho)=0,
\end{equation}
or, in terms of commutators
\begin{equation}\label{eq:smoothmomentum}
\begin{aligned}
\partial_t(\rho^\eps u^\eps)+\diverg((\rho u)^\eps\otimes u^\eps)+\nabla p(\rho^{\eps})
&=\partial_t(\rho^\eps u^\eps - (\rho u)^\eps) + \diverg((\rho u)^\eps\otimes u^\eps\\& - (\rho u\otimes u)^\eps)
+ \nabla\left(p(\rho^\eps) - p^\eps(\rho)\right).
\end{aligned}
\end{equation}	
Making use of the following identity
\begin{equation}
\diverg((\rho u)^\eps\otimes u^\eps) = u^\eps\diverg{(\rho u)^\eps} + ((\rho u)^\eps\cdot\nabla)u^\eps,
\end{equation}
we can see that multiplying~\eqref{eq:smoothmomentum} by $u^\eps$ yields
\begin{equation}\label{eq:smoothmomentum2}
\rho^\eps\partial_t\left(\frac12|u^\eps|^2\right)+\left((\rho u)^\eps\cdot\nabla\right)\frac12|u^\eps|^2 + \rho^\eps u^\eps\;\nabla\left(P'(\rho^\eps)\right)
=r_1^\eps + r_2^\eps + r_3^\eps,
\end{equation}
where
\begin{align}
r_1^\eps &= \partial_t(\rho^\eps u^\eps - (\rho u)^\eps)\cdot u^\eps,\\
r_2^\eps &= \diverg((\rho u)^\eps\otimes u^\eps - (\rho u\otimes u)^\eps)\cdot u^\eps,\\
r_3^\eps &= \nabla\left(p(\rho^\eps) - p^\eps(\rho)\right)\cdot u^\eps. \label{eq:R3}
\end{align}

\noindent Using the mollified continuity equation
\begin{equation}\label{eq:smoothmass}
\partial_t\rho^\eps+\diverg(\rho u)^\eps=0,
\end{equation}
multiplied by $\frac12|u^\eps|^2$, we can rewrite~\eqref{eq:smoothmomentum2} as
\begin{equation}\label{eq:smoothmomentum3}
\partial_t\left(\frac12\rho^\eps|u^\eps|^2\right) + \diverg\left((\rho u)^\eps\;\frac12|u^\eps|^2\right) + \rho^\eps u^\eps\;\nabla\left(P'(\rho^\eps)\right)
=r_1^\eps + r_2^\eps + r_3^\eps.
\end{equation}
On the other hand writing~\eqref{eq:smoothmass} in the form
\begin{equation}\label{eq:smoothmass2}
\partial_t\rho^\epsilon+\diverg(\rho^\epsilon u^\epsilon)=\diverg(\rho^\epsilon u^\epsilon-(\rho u)^\epsilon),
\end{equation}
and multiplying by $P'(\rho^\eps)$ we get
\begin{equation}\label{eq:smoothmassP}
\partial_t\left(P(\rho^\eps)\right)+\diverg(\rho^\eps u^\eps)P'(\rho^\eps) = \diverg(\rho^\eps u^\eps - (\rho u)^\eps)\;P'(\rho^\eps).
\end{equation}
Combining~\eqref{eq:smoothmomentum3} and~\eqref{eq:smoothmassP} we obtain
\begin{equation}\label{eq:smoothEnergy}
\begin{aligned}
\partial_t\left(\frac12\rho^\eps |u^\eps|^2 + P(\rho^\eps)\right) &+ \diverg{\left((\rho u)^\eps\;\frac12|u^\eps|^2+\rho^\eps u^\eps P'(\rho^\eps)\right)}\\
&\hspace{-0.5cm}=r_1^\eps + r_2^\eps + r_3^\eps + s^\eps,
\end{aligned}
\end{equation}
where we set
\begin{equation}
s^\eps \coloneqq \diverg(\rho^\eps u^\eps - (\rho u)^\eps)\;P'(\rho^\eps).
\end{equation}

\medskip

The proof of Theorem~4.1 in \cite{FGSW} shows that when $\rho,u$ are Besov regular and $p$ is of class $C^2$, then the left-hand side of~\eqref{eq:smoothEnergy} converges to the left-hand side of~\eqref{eq:energyidentity} and each term on the right-hand side of~\eqref{eq:smoothEnergy} converges to zero, each convergence in the sense of distributions.


\section{Energy Conservation assuming the divergence of velocity is a  bounded measure}\label{section:divergencevmeasure}

Our first result establishes local energy conservation for weak solutions of~\eqref{eq:compressibleEuler} under the additional assumption that the velocity field $u$ is a divergence-measure field. 

\begin{remark}
	See ~\cite{ChenTorres2005}, and references therein, for details on the role of divergence-measure fields in the theory of hyperbolic conservation laws.
\end{remark}

\begin{theorem}\label{thm:compressibleonsagerDivMeas}
	Let $\rho$, $u$ be a solution of~\eqref{eq:compressibleEuler} in the sense of distributions. Assume 
	\begin{equation}
	u\in B_3^{\alpha,\infty}((0,T)\times\Tbb^d),\hspace{0.3cm}\rho, \rho u\in B_3^{\beta,\infty}((0,T)\times\Tbb^d),\hspace{0.3cm}
	0 \leq \underline{\rho} \leq \rho \leq \overline{\rho} \ \mbox{a.e. in } (0,T)\times\Tbb^d,
	\end{equation}
	for some constants $\underline{\vr}$, $\overline{\vr}$, and
	$0\leq\alpha,\beta\leq1$ such that
	\begin{equation}\label{alphabeta}
	\beta > \max \left\{ 1 - 2 \alpha; \frac{1 - \alpha}{2} \right\}.
	\end{equation}
	Assume further that
	\begin{equation}
	\diverg u \in \mathcal{M}((0,T)\times \T^d),\qquad\text{and}\qquad
	p \in C[\underline{\vr}, \overline{\vr}].
	\end{equation}
	Then the energy is locally conserved, i.e.
	\begin{equation}
	\partial_t\left(\frac{1}{2}\rho|u|^2+P(\rho)\right)+\diverg\left[\left(\frac{1}{2}\rho|u|^2+p(\rho)+P(\rho)\right)u\right]=0
	\end{equation}
	in the sense of distributions on $(0,T)\times\Tbb^d$.
\end{theorem}

\begin{proof}
	Take a sequence $p^\delta\in C^2[\underline{\rho}, \overline{\rho}]$ that converges uniformly to $p\in C[\underline{\rho}, \overline{\rho}]$, that is, for each $\delta>0$
	\begin{equation}
	\|p-p^\delta\|_{L^\infty}\le \delta. 
	\end{equation}
	Then using $p^\delta$ in~\eqref{eq:mollifiedmomentum} we have 
	\begin{equation}\label{eq:pdeltamollifiedsplit}
	\partial_t(\rho u)^\eps+\diverg(\rho u\otimes u)^\eps+\nabla (p^\delta(\rho))^\eps=\nabla [(p^\delta(\rho))^\eps-p^\eps(\rho)]. 
	\end{equation}
	Now the left-hand side of the last equality satisfies all the conditions of Theorem~\ref{thm:compressibleonsager}, so for each fixed $\delta>0$ we have, in the limit as $\eps\to0$, 
	\begin{equation}\label{eq:Localenergydeltaequation}
	\partial_t\left(\frac{1}{2}\rho|u|^2+P^\delta(\rho)\right) + \diverg\left[\left(\frac{1}{2}\rho|u|^2+p^\delta(\rho)+P^\delta(\rho)\right)u\right] ,
	\end{equation}
	where
	\begin{equation}
	P^\delta(\rho)\coloneqq \rho\int_1^\rho\frac{p^\delta(r)}{r^2}dr.
	\end{equation}
	We will now show that~\eqref{eq:Localenergydeltaequation} converges as $\delta\to 0$ in the sense of distributions on $(0,T)\times \T^d$ to
	\begin{equation}
	\partial_t\left(\frac{1}{2}\rho|u|^2+P(\rho)\right)+\diverg\left[\left(\frac{1}{2}\rho|u|^2+p(\rho)+P(\rho)\right)u\right].
	\end{equation}
	Let $\phi\in\DC$.  From the choice of $p^\delta$ we have
	\begin{equation}\label{eq:pdeltabound}
	\left|\int^T_0\intT{\nabla \phi \cdot (p^\delta(\rho)-p(\rho))u}\dt\right|\leq
	C\|\phi\|_{\Ccal^1}\|p^\delta-p\|_{L^\infty}\|u\|_{L^3}\le C(\phi,u)\delta.
	\end{equation}
	
	\noindent For the terms containing $P^\delta(\rho)$ notice that 
	\begin{equation}
	|P^\delta(\rho)-P(\rho)|\le \rho \int^{\rho}_1\frac{|p^\delta(r) -p(r)|}{r^2}\d r\le\|p^\delta -p\|_{L^\infty}\;\rho \left|\int^{\rho}_1\frac{1}{r^2}\d r\right|\leq(1+\rho)\|p^\delta -p\|_{L^\infty}. 
	\end{equation}
	Hence we can estimate
	\begin{equation}
	\left|\int_0^T\intT{\partial_t\phi\; (P^\delta(\rho)-P(\rho))}dt\right|\leq C\norm{\phi}_{\Ccal^1}(1+\norm{\rho}_{L^1})\delta\le C(\phi)\delta,
	\end{equation}
	and similarly for the divergence term. It follows that both terms of~\eqref{eq:Localenergydeltaequation} containing $P^\delta$ converge as $\delta\to0$ to the corresponding terms for $P$.
	
	The final step of the proof is to consider the term coming into~\eqref{eq:smoothEnergy} from the right-hand side of~\eqref{eq:pdeltamollifiedsplit}. We need to show that
	\[
	\nabla [(p^\delta(\rho))^\eps-p^\eps(\rho)] \cdot u^\eps
	\]
	converges to zero in the sense of distributions on $(0,T)\times\T^d$ as first $\eps $ and  then $\delta$ tend to zero. Multiplying by $\phi\in\DC$, integrating over time and space, and integrating by parts we obtain
	\begin{align}\label{eq:pressurecommutator}
	\int^T_0\int_{\T^d}\nabla [(p^\delta(\rho))^\eps-p^\eps(\rho)]\phi u_\eps \d x\d t=-&\int^T_0\int_{\T^d} [(p^\delta(\rho))^\eps-p^\eps(\rho)]\phi  \diverg u^\eps \d x\d t
	\\-&\int^T_0\int_{\T^d} [(p^\delta(\rho))^\eps-p^\eps(\rho)]\nabla\phi\cdot u^\eps \d x\d t.
	\end{align}
	For the second term on the right-hand side of the last equality we see that 
	\begin{align}
	\left|\int^T_0\int_{\T^d} [(p^\delta(\rho))^\eps-p^\eps(\rho)]\nabla\phi\cdot u^\eps \d x\d t\right|&=\left|\int^T_0\int_{\T^d} [p^\delta(\rho)-p(\rho)]^\eps\nabla\phi\cdot u^\eps \d x\d t\right|
	\\&\le C\|\phi\|_{C^1}\|(p^\delta-p)^\eps\|_{L^\infty}\|u\|_{L^3}
	\\&\le C\|\phi\|_{C^1}\|p^\delta-p\|_{L^\infty}\|u\|_{L^3}\le C\delta.
	\end{align}
	Finally, for the first term on the right-hand side of~\eqref{eq:pressurecommutator} we invoke the assumption that $\diverg u$ is a bounded Radon measure to see that  
	\begin{align}
	\left|\int^T_0\int_{\T^d} \phi[(p^\delta(\rho))^\eps-p^\eps(\rho)] \diverg u^\eps \d x\d t\right|&= \left|\int^T_0\int_{\T^d} \phi[p^\delta(\rho)-p(\rho)]^\eps (\diverg u)^\eps \d x\d t\right|
	\\&\le\|\phi\|_{C^0}\|(p^\delta-p)^\eps\|_{L^\infty}\|(\diverg u)^\eps \|_{L^1}
	\\&\le\|\phi\|_{C^0}\|p^\delta-p\|_{L^\infty}\|\diverg u \|_{TV}\le C \delta
	\end{align}
	and so we are done.
\end{proof}


\subsection{Application to the compressible Navier-Stokes equations}\label{NSE}

When studying the result of Theorem \ref{thm:compressibleonsagerDivMeas}  we see that the condition 
$	\diverg u \in \mathcal{M}((0,T)\times \T^d)$  is  quite a strong  assumption for solutions to the compressible Euler equations, 
however, it is given for the compressible Navier-Stokes equations where one obtains a-priori from the diffusion term that $u\in L^2(0,T;H^1)$. Therefore a natural question to ask is what happens when we consider the solutions to the compressible Navier-Stokes   equations with vacuum, and how these results relate to the current results by Yu in \cite{Yu}. 

The compressible Navier-Stokes equations  are given by 
\begin{align}\label{eq:compressibleNavierStokes}
\partial_t(\rho u)+\diverg(\rho u\otimes u)+\nabla p(\rho)&=\diverg \mathbb{S}(\nabla u),\\
\partial_t\rho+\diverg(\rho u)&=0,\\
\mathbb{S}(\nabla u):= \mu\Big (\nabla u +(\nabla u)^T&-\frac{2}{3}\diverg u \mathbb{I}\Big )+\nu \diverg u \mathbb{I}
\end{align}
where we have the constants $\nu>0$ and $\eta\ge 0$. Here we will use the main properties that $\mathbb{S}(\nabla u)$ is symmetric and positive definite.
For degenerate viscosity, the momentum equation becomes, instead,
\begin{equation}\label{eq:compressibleNavierStokesDegenerate}
\partial_t(\rho u)+\diverg(\rho u\otimes u)+\nabla p(\rho)=\diverg(\rho \mathbb{S}(\nabla u)).
\end{equation}

\begin{corollary}\label{corollary:compressibleNavierStokes}
		Let $\rho$, $u$ be a solution of~\eqref{eq:compressibleNavierStokes} or ~\eqref{eq:compressibleNavierStokesDegenerate} in the sense of distributions. Assume 
	\begin{align}
	&u\in B_3^{\alpha,\infty}((0,T)\times\Tbb^d),\hspace{0.3cm} u \in L^2(0,T;H^1(\T^d)),\hspace{0.3cm}\rho, \rho u\in B_3^{\beta,\infty}((0,T)\times\Tbb^d),\\
	&0 \leq \underline{\rho} \leq \rho \leq \overline{\rho} \ \mbox{a.e. in } (0,T)\times\Tbb^d,
	\end{align}
	for some constants $\underline{\vr}$, $\overline{\vr}$, and
	$0\leq\alpha,\beta\leq1$ such that
	\begin{equation}\label{eq:alphabeta}
	\beta > \max \left\{ 1 - 2 \alpha; \frac{1 - \alpha}{2} \right\}.
	\end{equation}
	Assume further that $p \in C[\underline{\vr}, \overline{\vr}]$.
%
	Then the energy is locally conserved, i.e.
	\begin{equation}\label{eq:NavierStokesenergyidentity}
	\partial_t\left(\frac{1}{2}\rho|u|^2+P(\rho)\right)\hspace{-0.01cm}+\hspace{-0.01cm}\mathbb{S}(\nabla u): \nabla u\hspace{-0.01cm}+\hspace{-0.01cm}
	\diverg\left[\left(\frac{1}{2}\rho|u|^2+p(\rho)+P(\rho)+\mathbb{S}(\nabla u)\right)u\right]=0,
	\end{equation}
	 for \eqref{eq:compressibleNavierStokes} and 
	 \begin{multline}\label{eq:DegenerateNavierStokesenergyidentity}
	 \partial_t\left(\frac{1}{2}\rho|u|^2+P(\rho)\right)+\rho\mathbb{S}(\nabla u): \nabla u\\+
	 \diverg\left[\left(\frac{1}{2}\rho|u|^2+p(\rho)+P(\rho)+\rho\mathbb{S}(\nabla u)\right)u\right]=0,
	 \end{multline}
	 for \eqref{eq:compressibleNavierStokesDegenerate},
	 in the sense of distributions on $(0,T)\times\Tbb^d$.
\end{corollary}

\begin{remark}
	The condition $\diverg u \in \mathcal{M}$ is trivially satisfied if we assume that $u\in L^2(0,T;H^1)$ and so does not appear in the statement of Corollary \ref{corollary:compressibleNavierStokes}.
\end{remark}

\begin{remark}
	For $d\le 3$ we can use Besov embedding theorems, see \cite{bahouri2011fourier}, to observe that $H^1\hookrightarrow B^{1,\infty}_2\hookrightarrow B_3^{\frac{2}{3},\infty}$ and so assuming that $u\in B^{\alpha_1,\infty}_3(0,T;B^{\alpha_2,\infty}_3)$ and $\rho, \rho u\in B_3^{\beta_1,\infty}(0,T;B_3^{\beta_2,\infty})$ we have the same assumptions  on the pairs $(\alpha_1,\beta_1)$ and $(\alpha_2,\beta_2)$ as \eqref{eq:alphabeta} but can assume that $\alpha_2\ge\frac{2}{3}$ and remove the assumption that $u\in L^2(0,T;H^1)$. 
\end{remark}


\begin{proof}
We only have to consider the extra term
$\diverg \mathbb{S}(\nabla u)$ in the derivation of the local energy equality that we performed previously. We see that
\begin{multline}
-\int^T_0\int_{\T^d} \diverg \mathbb{S}(\nabla u^\eps)\cdot u^\eps \phi \d x \d t= \int^T_0\int_{\T^d}  \mathbb{S}(\nabla u^\eps):\nabla u^\eps \phi \d x \d t\\
+\int^T_0\int_{\T^d}  (\mathbb{S}(\nabla u^\eps)u^\eps)\cdot \nabla \phi \d x \d t
\end{multline}
and so obtain \eqref{eq:NavierStokesenergyidentity}. For \eqref{eq:DegenerateNavierStokesenergyidentity} we perform the same calculation as above however with an extra $\rho$ in the equation, the diffusion term is no longer linear and thus we pick up an extra commutator estimate
\begin{equation}
r_d^\eps:=\int^T_0\int_{\T^d}\diverg (\rho^\eps\mathbb{S}(\nabla u^\eps)-(\rho\mathbb{S}(\nabla u))^\eps)\cdot\phi u^\eps\d x\d t.
\end{equation}  
We can perform an integration by parts to obtain
\begin{align}\label{eq:degenerateviscositycomutator}
|r_d^\eps|&\le \left|\int^T_0\int_{\T^d} [(\rho^\eps\mathbb{S}(\nabla u^\eps)-(\rho\mathbb{S}(\nabla u))^\eps) u^\eps]\cdot\nabla\phi\d x\d t\right|\\
&+\left|\int^T_0\int_{\T^d} (\rho^\eps\mathbb{S}(\nabla u^\eps)-(\rho\mathbb{S}(\nabla u))^\eps) :\nabla u^\eps\phi\d x\d t\right|.
\end{align}
Note
the pointwise identity where for any two functions $f,g$ we have that
\begin{equation}\label{eq:pointwisedecomp}
\begin{aligned}
f^\eps g^\eps&-(f g)^\eps= (f^\eps-f)(g^\eps-g)\\
&-\int_{-\eps}^\eps\int_{\T^d}\eta^\eps(\tau,\xi)(f(t-\tau,x-\xi)-f(t,x))(g(t-\tau,x-\xi)-g(t,x)) d\xi d\tau.
\end{aligned}
\end{equation}
Applying this allows us to split the two terms on the R.H.S. of \eqref{eq:degenerateviscositycomutator} into four more terms which we can estimate.  
We focus on the first of these terms only, as the other terms produce the same estimates, after applying Fubini's theorem, as seen in \cite{FGSW}. We see that
\begin{align} 
|r_d^\eps|&\le \left|\int^T_0\int_{\T^d} [(\rho^\eps-\rho)(\mathbb{S}(\nabla u^\eps)-\mathbb{S}(\nabla u)) u^\eps]\cdot\nabla\phi\d x\d t\right|\\
&+\left|\int^T_0\int_{\T^d} (\rho^\eps-\rho)(\mathbb{S}(\nabla u^\eps)-\mathbb{S}(\nabla u)) :\nabla u^\eps\phi\d x\d t\right|
\\
&\le \|\phi\|_{C^1}\|\rho\|_{L^\infty}\|u\|_{L^2}\|\mathbb{S}(\nabla u^\eps)-\mathbb{S}(\nabla u)\|_{L^2}\\ &+\|\phi\|_{C^0}\|\rho\|_{L^\infty}\|\nabla u\|_{L^2}\|\mathbb{S}(\nabla u^\eps)-\mathbb{S}(\nabla u)\|_{L^2}.
\end{align}
Using the a-priori estimate that $u\in L^2(0,T;H^1)$ we see that 
$\|\mathbb{S}(\nabla u^\eps)-\mathbb{S}(\nabla u)\|_{L^2}\to 0$ as $\eps\to 0$ and thus $r_d^\eps\to 0$ as $\eps\to 0$.
\end{proof}

The  work of Cheng Yu in \cite{Yu} also studies energy conservation for the compressible Navier-Stokes systems where a vacuum could occur. The result in \cite{Yu} treats the case where $p(\rho)=\rho^\gamma$ for $\gamma>1$ and thus where $p\in C^{1,\gamma-1}$, with strong assumptions of spacial regularity where
\begin{equation}
\sqrt{\rho}\nabla u \in L^2(0,T;L^2(\Omega)) \quad \text{and} \quad \frac{\nabla \rho}{\sqrt{\rho}}\in L^\infty(0,T;L^2(\Omega))
\end{equation}
among other assumptions, see \cite{Yu} for more details. However, \cite{Yu} only assumes integrability in time. The condition $\frac{\nabla \rho}{\sqrt{\rho}}\in L^\infty(0,T;L^2(\Omega))$ restricts the allowable vacuum cases and will only allow vacuum on measure zero sets with a nice approach to this set. The result presented here complements the result in \cite{Yu} as we show that by assuming some differential regularity in time for both $\rho$ and $u$ then we can weaken the spacial regularity assumptions and only need continuity of the pressure $p$. Specifically, we can have vacuum on measurable subsets of the domain where the approach to this set can be quite generic.

\section{Energy Conservation assuming H\"{o}lder continuity of the pressure}\label{section:Holdercontinuity}

\medskip
For the next result we fix $1<\gamma<2$ and we will assume that the pressure $p$ is of class $C^{1,(\gamma-1)}$, 
thus relaxing the regularity assumption of Theorem~\ref{thm:compressibleonsager}. The expense of this relaxation is that we require $\alpha +\gamma \beta>1$ where before we only needed $\alpha +2\beta>1$.

\begin{theorem}\label{thm:compressibleonsagerCGamma}
	Let $\rho$, $u$ be a solution of~\eqref{eq:compressibleEuler} in the sense of distributions. Assume 
	\begin{align}
	u&\in B_p^{\alpha,\infty}((0,T)\times\T^d),\hspace{0.3cm}\rho,\rho u\in B_q^{\beta,\infty}((0,T)\times\T^d),\\
	0 &\leq \underline{\rho} \leq \rho \leq \overline{\rho} \ \mbox{a.e. (t,x) in}\; (0,T)\times\Tbb^d, 
	\end{align}
	for some constants $\underline{\vr}$, $\overline{\vr}$ and $0\leq\alpha,\beta\le1$ such that,
	\begin{equation}
	\frac1p + \frac2q \le 1,\qquad \frac2p + \frac1q \le 1,\qquad p,q\ge 2,  \qquad 	\alpha+\gamma\beta>1\qquad \text{and}\qquad
	2\alpha+\beta>1.
	\end{equation}
	Define  $\mathcal{B}_{\eps^\beta}:=\{x:0<\rho^\eps(x)<\eps^\beta \text { and } \rho\neq0\}$ and assume that 
	\begin{equation}\label{eq:mollifierratiocondition}
 \left\|\frac{\rho^\eps-\rho}{\rho^\eps}\right\|_{L^q(\mathcal{B}_{\eps^\beta})}\le C(\rho),
	\end{equation}
	where $C$ does not depend on $\eps$. Assume further that $p\in C^{1,(\gamma-1)}([\underline{\rho},\overline{\rho}])$, and, in addition 
	\begin{equation}
	p'(0)=0\quad \text{as soon as}\quad \underline{\rho}=0.
	\end{equation}
	Then the energy is locally conserved, i.e.~\eqref{eq:energyidentity} holds
	in the sense of distributions on $(0,T)\times\Tbb^d$.
\end{theorem}


Large part of the proof of this theorem is identical to the proof of Theorem \ref{thm:compressibleonsager}. In particular we regularize the balance equations to derive an energy balance for the smooth functions $\rho^\eps$ and $u^\eps$. Then we need to show that the corresponding commutator errors vanish in the limit $\eps\to0$. This is done in the same way as in~\cite{FGSW}, the only difference being in the terms involving the pressure. In particular, we will have to estimate an appropriate norm of the difference $p(\rho)^\eps-p(\rho^\eps)$. This will be done by means of the following lemma, which is an adaptation to our present case of the argument in~\cite[p.~10]{FGSW}, see also~\cite[Lemma~3.1]{GMSG}.

\begin{lemma}\label{lemma:Cgamma}
	Let $\gamma\in(1,2)$ and $p\in C^{1,\gamma-1}([a,b])$. If $\rho\in B_q^{\beta,\infty}(\Omega;[a,b])$, then
	\begin{equation}\label{eq:CGammaCommutatorEstimate}
	\norm{p^\eps(\rho)-p(\rho^\eps)}_{L^q}\leq C\eps^{\gamma\beta}\norm{\rho}^{\gamma}_{B_q^{\beta,\infty}}
	\end{equation}
\end{lemma}

\begin{proof}
First we note that by the fundamental theorem of calculus
\begin{align}
p(s)-p(s_0)=\int^s_{s_0} p'(t) \d t&=\int^s_{s_0} p'(s_0) \d t+\int^s_{s_0} p'(t)-p'(s_0) \d t\\
&=p'(s_0)(s-s_0)+\int^s_{s_0} p'(t)-p'(s_0) \d t.
\end{align}
Since $p'\in C^{0,\gamma-1}$, we have  
\begin{equation}
\left|\int^s_{s_0} p'(t)-p'(s_0) \d t\right|\le \int^s_{s_0} |p'(t)-p'(s_0)| \d t\le C\int^s_{s_0}\d t \sup_{t\in[s_0,s]}|t-s_0|^{\gamma-1}\le C|s-s_0|^{\gamma}.
\end{equation}
Thus, 
\begin{equation}
|p(s)-p(s_0)-p'(s_0)(s-s_0)|\le C|s-s_0|^{\gamma}.
\end{equation}
As the constant $C$ is independent of $s,s_0$ we see that 
\begin{equation}\label{eq:Taylor1}
|p(\rho^\eps)-p(\rho)-p'(\rho)(\rho^\eps-\rho)|\le C|\rho-\rho^\eps|^{\gamma},
\end{equation}
and similarly,
\begin{equation}\label{eq:Taylor2}
|p(\rho(y))-p(\rho(x))-p'(\rho(x))(\rho(y)-\rho(x))|\le C|\rho(x)-\rho(y)|^{\gamma}.
\end{equation}
 We can apply convolution against the function $\eta^\eps$ with respect to $y$ in \eqref{eq:Taylor2} and apply Jensen's inequality over the convolution integral to obtain
 \begin{equation}\label{eq:Taylor3}
 |p^\eps(\rho)-p(\rho)-p'(\rho)(\rho^\eps-\rho)|\le C|\rho-\rho(\cdot)|^{\gamma}\ast_y\eta^\eps.
 \end{equation} 
Combining~\eqref{eq:Taylor1} and~\eqref{eq:Taylor3} we get
 \begin{equation}\label{eq:pHoldersbound}
 |p^\eps(\rho)-p(\rho^\eps)|\le C|\rho-\rho^\eps|^{\gamma}+ C|\rho-\rho(\cdot)|^{\gamma}\ast_y\eta^\eps.
 \end{equation}
Taking the $L^q$ norm of both sides of  \eqref{eq:pHoldersbound} for the first term on the R.H.S. we see that
\begin{equation}
C\||\rho-\rho^\eps|^{\gamma}\|_{L^q}= C\|\rho-\rho^\eps\|^\gamma_{L^{\gamma q}}.
\end{equation}
Finally,  for the $L^q$ norm of  \eqref{eq:pHoldersbound} for the second term on the R.H.S. by Jensen's inequality and Fubini's theorem we have
\begin{align}
C\||\rho-\rho(\cdot)|^{\gamma}\ast_y\eta^\eps\|_{L^q}&\le C\left(\int \int |\rho(x)-\rho(x-y)|^{\gamma q} \d x  \eta_\eps (y)\d y\right)^{1/q}\\
&=C\left(\int \|\rho(\cdot)-\rho(\cdot-y) \|^{\gamma q}_{L^{\gamma q}}  \eta_\eps (y)\d y\right)^{1/q}\\
&\le C\sup_y |\eta_\eps (y)|^{1/q}\left(\int_{\supp \eta_\eps} \|\rho(\cdot)-\rho(\cdot-y) \|^{\gamma q}_{L^{\gamma q}} \d y\right)^{1/q}\\
&\le C \sup\limits_{y\in\supp\eta_\eps}\norm{\rho(\cdot)-\rho(\cdot-y)}^\gamma_{L^{\gamma q}}.
\end{align}

Finally, we use the definition of the Besov norm and~\eqref{besovshift} to write
\begin{equation}
\begin{aligned}
\norm{p^\eps(\rho)-p(\rho^\eps)}_{L^q}&\leq C\left(\norm{\rho^\eps-\rho}_{L^{\gamma q}}^\gamma + \sup\limits_{s\in\supp\eta^\eps}\norm{\rho(\cdot)-\rho(\cdot-s)}^\gamma_{L^{\gamma q}}  \right)\\
&\leq C\eps^{\gamma\beta}\norm{\rho}^{\gamma}_{B_{\gamma q}^{\beta,\infty}} + \sup\limits_{s\in\supp\eta^\eps}|s|^{\gamma\beta}\norm{\rho}^{\gamma}_{B_{\gamma q}^{\beta,\infty}}\leq C\eps^{\gamma\beta}\norm{\rho}^{\gamma}_{B_q^{\beta,\infty}},
\end{aligned}
\end{equation}
using that $q\gamma < q$, so that $B_q^{\beta,\infty}\subset B_{\gamma q}^{\beta,\infty}$.
\end{proof}

\begin{proof}[Proof of Theorem~\ref{thm:compressibleonsagerCGamma}]
As remarked above the only novelty needed to establish the desired result is to estimate commutator errors due to nonlinearity of the pressure. Precisely, we need to show that the local version of $r_3^\eps$ and $s^\eps$, which we will denote $R_3^\eps$ and $S^\eps$, of equation~\eqref{eq:smoothEnergy} converge to zero as $\eps\to0$. For a test function $\phi\in\DC$ we denote

\begin{equation}\label{eq:errorR}
R^\eps:=\int^T_0\int_{\T^d}\nabla (p(\rho^\eps)-p(\rho)^\eps)\cdot\phi u^\eps\d x\d t,
\end{equation}
and 
\begin{equation}\label{eq:errorS}
S^\eps := \int^T_0\int_{\T^d}\phi \diverg [\rho^\eps u^\eps-(\rho u)^\eps)]\;P'(\rho^\eps)\d x\d t.
\end{equation}
Integrating~\eqref{eq:errorR} by parts and using Lemma~\ref{lemma:Cgamma} we obtain the following  estimate.
\begin{equation}
\begin{aligned}
|R^\eps|&\leq\norm{\phi}_{\Ccal^1}\int_0^T\intT{|p(\rho)^\eps-p(\rho)^\eps|(|\nabla u^\eps|+|u^\eps|)}\dt\\
&\leq C\norm{\phi}_{\Ccal^1}\norm{p(\rho^\eps)-p(\rho)^\eps}_{L^{\frac q2}}(\norm{\nabla u^\eps}_{L^p}+\norm{u^\eps}_{L^p})\\
&\leq C\norm{\rho^\eps-\rho}^\gamma_{L^\frac{\gamma q}{2}}(\norm{\nabla u^\eps}_{L^p}+\norm{u^\eps}_{L^p})\\
&\leq C(\eps^{\gamma\beta+(\alpha-1)}+\eps^{\gamma\beta+\alpha})\norm{\rho}^\gamma_{B_{q}^{\beta,\infty}}\norm{u}_{B_p^{\alpha,\infty}}
\end{aligned}
\end{equation}
We note that $\frac{\gamma q}{2}< q$, so we can embed $B_q^{\beta,\infty}$ into $B_{\frac{\gamma q}2}^{\beta,\infty}$.

\medskip
We now investigate the term $S^\eps$ and see that we can integrate by parts to obtain 
\begin{multline}\label{eq:sepssplit1}
|S^\eps|= \left|\int^T_0\int_{\T^d}\phi \diverg [\rho^\eps u^\eps-(\rho u)^\eps)]\;P'(\rho^\eps)\d x\d t\right|\\\le \int^T_0\int_{\T^d}|\nabla \phi \cdot [\rho^\eps u^\eps-(\rho u)^\eps)]\;P'(\rho^\eps)|\d x\d t+\int^T_0\int_{\T^d}|\phi  [\rho^\eps u^\eps-(\rho u)^\eps)]\cdot  \nabla P'(\rho^\eps)|\d x\d t.
\end{multline}
 We make note of the following pointwise identity \eqref{eq:pointwisedecomp} but with $f$ and $g$ replaced by $\rho$ and $u$ respectively, that is, 
\begin{equation}
\begin{aligned}
&\rho^\eps u^\eps-(\rho u)^\eps= (\rho^\eps-\rho)(u^\eps-u)\\
&\hspace{0.5cm}-\int_{-\eps}^\eps\int_{\T^d}\eta^\eps(\tau,\xi)(\rho(t-\tau,x-\xi)-\rho(t,x))(u(t-\tau,x-\xi)-u(t,x)) d\xi d\tau
\end{aligned}
\end{equation} 
and using~\eqref{eq:pointwisedecomp} allows us to split first term on the R.H.S. of \eqref{eq:sepssplit1} into two terms. Here again we focus on the first of these terms only, as the other one produces the same estimates, after applying Fubini's theorem, as seen in \cite{FGSW}.
We see that
\begin{equation}
\int^T_0\int_{\T^d}|\nabla\phi \cdot  (\rho^\eps-\rho)(u^\eps-u)P'(\rho^\eps)|\d x\d t\le \|\phi\|_{C^1}\eps^\beta\|\rho\|_{B_q^{\beta,\infty}} \eps^\alpha\|u\|_{B_p^{\alpha,\infty}} \|P'(\rho^\eps)\|_{L^\infty}.
\end{equation}

We will now focus on the second term on the R.H.S.  of \eqref{eq:sepssplit1}, namely, 
\begin{equation}
\int^T_0\int_{\T^d}|\phi  [\rho^\eps u^\eps-(\rho u)^\eps)]\cdot  \nabla P'(\rho^\eps)|\d x\d t.
\end{equation}
and by letting $y=(x,t)$ we split $\T^d\times (0,T)$ into two disjoint domains $\mathcal{A}:=\{y:\rho^\eps(y)=0 \}$ and $\mathcal{A}^c$ and see that trivially on $\mathcal{A}$ that $\rho(y)= 0$ a.e.. 
For the integral over $\mathcal{A}$ we note that $\nabla P'(\rho^\eps)$  is a distribution that 
  may have a singular part but we see that $\phi  [\rho^\eps u^\eps-(\rho u)^\eps]$ is smooth and equals zero on $\mathcal{A}$ 
and so any singular part vanishes. Thus we are left with
\begin{equation}
\int_{\mathcal{A}^c}|\phi   [\rho^\eps u^\eps-(\rho u)^\eps]\nabla P'(\rho^\eps)|\d x\d t
\end{equation}
and using again the identity \eqref{eq:pointwisedecomp} we obtain
 \begin{equation}
 \int_{\mathcal{A}^c}|\phi   [(\rho^\eps-\rho)(u^\eps-u)]\nabla P'(\rho^\eps)|\d x\d t.
 \end{equation}

For the integral over $\mathcal{A}^c$ we see that 
\begin{equation}
\int_{\mathcal{A}^c}|\phi  (\rho^\eps-\rho)(u^\eps-u)\nabla P'(\rho^\eps)|\d x\d t=\int_{\mathcal{A}^c}|\phi  (\rho^\eps-\rho)(u^\eps-u) P''(\rho^\eps)\cdot \nabla \rho^\eps |\d x\d t
\end{equation}
and we observe that by the definition of $P$ we have $\rho^\eps P''(\rho^\eps) = p'(\rho^\eps)$, and by assumption $p'$ is bounded. Therefore we have the following bound
%
%
%
	\begin{multline}
	\int_{\mathcal{A}^c}\left|\phi(\rho^\eps-\rho)(u^\eps-u)P''(\rho^\eps)\nabla\rho^\eps\right| \d x\d t\\\leq \int_{\mathcal{A}^c}\left|\phi(\rho^\eps-\rho)(u^\eps-u)p'(\rho^\eps)\frac{\nabla \rho^\eps}{\rho^\eps}\right|\d x \d t.
	\end{multline}
	We have assumed that $p'(0)=0$ and $p'\in C^{0,\gamma-1}$ and so take any $\rho_1,\rho_2$ such that $p'(\rho_2)=0$ and we obtain that
	\begin{equation}
	|p'(\rho_1)|=|p'(\rho_1)-p'(\rho_2 )|\le C|\rho_1 -\rho_2|^{\gamma -1} =C|\rho_1|^{\gamma -1}
	\end{equation}
	using the definition of H\"older continuity.
	Thus letting $\rho_1=\rho^\eps (x)$ for each $x$ we see that $|p'(\rho^\eps )(x)|\le C|\rho^\eps |^{\gamma -1}(x)$
	and so we obtain
	\begin{multline}
\int_{\mathcal{A}^c}\left|\phi(\rho^\eps-\rho)(u^\eps-u)p'(\rho^\eps)\frac{\nabla \rho^\eps}{\rho^\eps}\right|\d x \d t\\\leq C \int_{\mathcal{A}^c}\left|\phi(\rho^\eps-\rho)(u^\eps-u)(\rho^\eps)^{\gamma-1}\frac{\nabla \rho^\eps}{\rho^\eps}\right|\d x \d t.
	\end{multline}
	We will split the integral over $\mathcal{A}^c$ further into different disjoint domains, 
	 $\mathcal{B}_{\eps^\beta}:=\{y:0<\rho^\eps(y)<\eps^\beta \}$ and $\mathcal{C}_{\eps^\beta}:=\{y:\rho^\eps(y)\ge\eps^\beta\}$.
	For the integral over $\mathcal{B}_{\eps^\beta}$ we see that 
	\begin{align}
	\left|\int_{\mathcal{B}_{\eps^\beta}}\phi(\rho^\eps-\rho)\right.&\hspace{-0.14cm}\left.(u^\eps-u)(\rho^\eps)^{\gamma-1}\frac{\nabla \rho^\eps}{\rho^\eps}\d x\d t\right|\\\le & C\|\phi\|_{C^0}\eps^{\beta-1+\alpha}\|\rho\|_{B^{\beta,\infty}_{q}}\|u\|_{B^{\alpha,\infty}_{p}}\|(\rho^\eps)^{\gamma-1}\|_{L^\infty(\mathcal{B}_{\eps^\beta})}\left\|\frac{\rho^\eps-\rho}{\rho^\eps}\right\|_{L^q(\mathcal{B}_{\eps^\beta})}\\\le & C\|\phi\|_{C^0}\eps^{\gamma\beta-1+\alpha}\|\rho\|_{B^{\beta,\infty}_{q}}\|u\|_{B^{\alpha,\infty}_{p}}\left\|\frac{\rho^\eps-\rho}{\rho^\eps}\right\|_{L^q(\mathcal{B}_{\eps^\beta})},
	\end{align}
	where for the last line as $\rho^\eps(y)\le \eps^\beta$ so  $(\rho^\eps(y))^{\gamma-1}\le \eps^{\beta(\gamma-1)}$ as ${\gamma-1}>0$.
	We also have the  assumption that $\left\|\frac{\rho^\eps-\rho}{\rho^\eps}\right\|_{L^q(\mathcal{B}_{\eps^\beta})}\le C$ and so we have the bound $C \eps^{\gamma\beta-1+\alpha}$ as wanted.  
	We are left with the integral over $\mathcal{C}_{\eps^\beta}$ and see that
	\begin{multline}
	\left|\int_{\mathcal{C}_{\eps^\beta}} \phi(\rho^\eps-\rho)(u^\eps-u)(\rho^\eps)^{\gamma-1}\frac{\nabla \rho^\eps}{\rho^\eps}\d x\d t\right|\\\le C\|\phi\|_{C^0}\eps^{\beta-1+\alpha}\|\rho\|_{B^{\beta,\infty}_{q}}\|u\|_{B^{\alpha,\infty}_{p}}\left\|\frac{\rho^\eps-\rho}{(\rho^\eps)^{2-\gamma}}\right\|_{L^q(\mathcal{C}_{\eps^\beta})}.
	\end{multline}
	As $\rho^\eps\ge\eps^\beta $ thus $(\rho^\eps)^{-1}\le \eps^{-\beta}$ and so $(\rho^\eps)^{\gamma-2}\le \eps^{\beta(\gamma-2)}$ we obtain
	\begin{equation}
	\left\|\frac{\rho^\eps-\rho}{(\rho^\eps)^{2-\gamma}}\right\|_{L^q(\mathcal{C}_{\eps^\beta})}\le \left\|\rho^\eps-\rho\right\|_{L^q(\mathcal{C}_{\eps^\beta})}\eps^{\beta(\gamma-2)}\le C \eps^\beta \|\rho\|_{B^{\beta,\infty}_q}\eps^{\beta(\gamma-2)}\le C \eps^{\beta(\gamma-1)}. 
	\end{equation}
	We are thus done as obtain convergence to zero as long as $\gamma\beta+\alpha >1$.
%
%

We have thus shown that, under the assumptions of the theorem, we have $R^\eps, S^\eps\to0$ . The result follows.
\end{proof}


We have written Theorem \ref{thm:compressibleonsagerCGamma} in the most general from but observe that the condition
\begin{equation}
\left\|\frac{\rho^\eps-\rho}{\rho^\eps}\right\|_{L^q(\mathcal{B}_{\eps^\beta})}\le C 
\end{equation}
 feels rather artificial and is not in the $p\in C^2$ result from \cite{FGSW}. We will now focus on finding conditions on $\rho$ for different $L^q$ norms that will control this term. 

Our first result will show that when we assume that $q=1$ and so $u,\rho$ are H\"older continuous, not just Besov functions, then we can control this term directly as expected and not have to ask for any special extra conditions.     

\begin{lemma}\label{mollificationlem}
For a non-negative function $w\in L^1$ then on $\Omega\subset \T^N$ where $|\Omega|\neq 0$ and $w^\eps|_{\Omega}>0$ then $\left\|\frac{w^\eps-w}{w^\eps}\right\|_{L^1(\Omega)}\le C$ where $C$ does not depend on $\eps$ but may depend on $w$ and $\Omega$.
\end{lemma}


\begin{proof}
	Firstly, we notice that as $|\Omega|\le C$ so 
	\begin{equation}
	\left\|\frac{w^\eps-w}{w^\eps}\right\|_{L^1(\Omega)}\le \|1\|_{L^1(\Omega)}+\left\|\frac{w}{w^\eps}\right\|_{L^1(\Omega)} =C+\left\|\frac{w}{w^\eps}\right\|_{L^1(\Omega)}
	\end{equation}
	and so we just need to show that $\left\|\frac{w}{w^\eps}\right\|_{L^1(\Omega)}\le C$.

For $w\in L^1(\Omega)$ we will perform the standard approximation by simple functions used in measure theory. We take the Borel $\sigma$-algebra $\mathcal{A}$ on $\Omega$ and let $A_i\in \mathcal{A}$ be pairwise-disjoint and $a_i\in \R_+$ and let $w_n=\sum_{i=1}^n a_i \chi_{A_i}$ where $\lim_{n\to \infty}\|w_n-w\|_{L^1(\Omega)}=0$ where $w_n\le w_m$ for all $n<m$. We see that $|A_i|\neq 0$ and $a_i>0$ as if not we could just remove these terms and this would not change the value of $\|w_n-w\|_{L^1(\Omega)}$.

We can cover $\Omega$ in balls $\{B_{\eps/4}(x_j)\}_{j=1}^m$ of radius $\eps/4$ so that $\sum^m_{j=1}|B_{\eps/4}(x_j)| \le 2^N|\Omega|$. We can then create a new collection of open sets by taking the refinement of $A_i\in \mathcal{A}$  with $\{B_{\eps/4}(x_j)\}^m_{j=1}$ and re-label  the new collection $\{D_j\}_{j=1}^\beta$ and so can write $w_n=\sum_{i=1}^n a_i \chi_{A_i}=\sum_{i=1}^k b_i \chi_{D_i}$. We can collect the terms together based on the sets  $\{B_{\eps/4}(x_j)\}^m_{j=1}$ where if the set is in multiple $B_{\eps/4}(x_j)$ we chose it to only appear in one of the sets and so obtain
 $w_n=\sum_{i=1}^k b_i \chi_{D_i}= \sum_{j=1}^m\sum_{i=1}^{k_j} b_i \chi_{D_i}$ and thus
 \begin{equation}
 \left\|\frac{w_n}{w_n^\eps}\right\|_{L^1(\Omega)}=\left\|\frac{\sum_{j=1}^m\sum_{i=1}^{k_j} b_i \chi_{D_i}}{\sum_{i=1}^k b_i \chi_{D_i}^\eps}\right\|_{L^1(\Omega)}\le \sum_{j=1}^m\left\|\frac{\sum_{i=1}^{k_j} b_i \chi_{D_i}}{\sum_{i=1}^k b_i \chi_{D_i}^\eps}\right\|_{L^1(\Omega)} \le
  \sum_{j=1}^m\left\|\frac{\sum_{i=1}^{k_j} b_i \chi_{D_i}}{\sum_{i=1}^{k_j} b_i \chi_{D_i}^\eps}\right\|_{L^1(\Omega)}. 
 \end{equation}
 We know that for every point in the ball $B_{\eps/4}(x_j)$ will be covered by the support of the mollifier of radius $\eps$ and so, ignoring the terms outside $B_{\eps/4}(x_j)$, we see that $\sum_{i=1}^{k_j} b_i \chi_{D_i}^\eps> C_d \frac{\sum_{i=1}^{k_j} b_i |{D_i}|}{|B_\eps|}$ for all $x\in B_{\eps/4}(x_j)$ and so   $\left(\sum_{i=1}^{k_j} b_i \chi_{D_i}^\eps\right)^{-1}< C_d \frac{|B_\eps|}{\sum_{i=1}^{k_j} b_i |{D_i}|}$ thus we can use this bound to obtain  
 \begin{align}\label{eq:L1cancleation}
 \sum_{j=1}^m\left\|\frac{\sum_{i=1}^{k_j} b_i \chi_{D_i}}{\sum_{i=1}^{k_j} b_i \chi_{D_i}^\eps}\right\|_{L^1(\Omega)}\le& \sum_{j=1}^m C_d \frac{|B_\eps|}{\sum_{i=1}^{k_j} b_i |{D_i}|}\left\|\sum_{i=1}^{k_j} b_i \chi_{D_i}\right\|_{L^1(\Omega)}\\=&\sum_{j=1}^m C_d \frac{|B_\eps|}{\sum_{i=1}^{k_j} b_i |{D_i}|}\sum_{i=1}^{k_j} b_i |D_i|
 =\sum_{j=1}^m C_d |B_\eps|\le \sum_{j=1}^m C_d |B_{\eps/4}|\le C_d |\Omega|. 
 \end{align}
 
 We have shown that $\left\|\frac{w_n}{w_n^\eps}\right\|_{L^1(\Omega)}\le C$ for any $n$ and any $\eps >0$. We know as well that as $n$ tends to infinity then $\frac{w_n}{w_n^\eps}\to \frac{w}{w^\eps}$ point-wise up to a subsequence  and so we can apply Fatou's lemma to say that 
 \begin{equation}
 \left\|\frac{w}{w^\eps}\right\|_{L^1(\Omega)}=\left\|\lim_{n\to \infty}\frac{w_n}{w_n^\eps}\right\|_{L^1(\Omega)} =\lim_{n\to \infty}\left\|\frac{w_n}{w_n^\eps}\right\|_{L^1(\Omega)} \le C
 \end{equation}  
 and so we are done.
\end{proof}

As a consequence we obtain the following corollary where by assuming H\"older continuity of $u$ and $\rho$ we obtain a natural extension of \ref{thm:compressibleonsager} to the case where $p\in C^{1,\gamma-1}$.

\begin{corollary}\label{HolderCor}
	Let $\rho$, $u$ be a solution of~\eqref{eq:compressibleEuler} in the sense of distributions. Assume 
	\begin{align}
	u&\in C^\alpha((0,T)\times\T^d),\hspace{0.3cm}\rho,\rho u\in C^\beta((0,T)\times\T^d),\\
	0 &\leq \underline{\rho} \leq \rho \leq \overline{\rho} \ \mbox{a.e. (t,x) in}\; (0,T)\times\Tbb^d, 
	\end{align}
	for some constants $\underline{\vr}$, $\overline{\vr}$ and $0\leq\alpha,\beta\le1$ such that,
	\begin{equation}
	 	\alpha+\gamma\beta>1\qquad \text{and}\qquad
	2\alpha+\beta>1.
	\end{equation}
	Assume further that $p\in C^{1,(\gamma-1)}([\underline{\rho},\overline{\rho}])$, and, in addition 
	\begin{equation}
	p'(0)=0\quad \text{as soon as}\quad \underline{\rho}=0.
	\end{equation}
		Then the energy is locally conserved, i.e.~\eqref{eq:energyidentity} holds
	in the sense of distributions on $(0,T)\times\Tbb^d$.
\end{corollary}

\begin{proof}
		For the integral over $\mathcal{B}_{\eps^\beta}$, in the proof of Theorem \ref{thm:compressibleonsagerCGamma}, we see that 
	\begin{align}
	\left|\int_{\mathcal{B}_{\eps^\beta}}\phi(\rho^\eps-\rho)\right.&\hspace{-0.14cm}\left.(u^\eps-u)(\rho^\eps)^{\gamma-1}\frac{\nabla \rho^\eps}{\rho^\eps}\d x\d t\right|\\\le &
	 C\|\phi\|_{C^0}\eps^{\beta-1+\alpha}\|\rho\|_{\mathcal{C}^\beta}\|u\|_{\mathcal{C}^\alpha}\|(\rho^\eps)^{\gamma-1}\|_{L^\infty(\mathcal{B}_{\eps^\beta})}\left\|\frac{\rho^\eps-\rho}{\rho^\eps}\right\|_{L^1(\mathcal{B}_{\eps^\beta})}\\
	 	\le & C\|\phi\|_{C^0}\eps^{\gamma\beta-1+\alpha}\|\rho\|_{\mathcal{C}^\beta}\|u\|_{\mathcal{C}^\alpha}.
	\end{align}
	For the other bounds as we are on a domain with finite measure so we can bound the Besov norms by the H\"older norms. 
\end{proof}

When we still want to consider Besov spaces for $\rho$ and $u$ we have to consider extra conditions on $\rho$ in order to control the term $\left\|\frac{\rho^\eps-\rho}{\rho^\eps}\right\|_{L^q(\mathcal{B}_{\eps^\beta})}$. Our first method will be to ask for an integrability condition on $\frac{1}{\rho}$.

\begin{lemma}\label{lemma:lprecipriocal}
	Assuming that $\frac{1}{w}\in L^p$ and $w \in L^q$ then 
	\begin{equation}
	\left\|\frac{w^\eps -w }{w^\eps}\right\|_{L^r}\le C\qquad \mathrm{for}\qquad \frac{1}{p}+\frac{1}{q}\le \frac{1}{r}  
	\end{equation} 
	and in fact if $r<\infty$
	\begin{equation}
	\lim_{\eps \to 0}\left\|\frac{w^\eps -w }{w^\eps}\right\|_{L^r} =0\qquad \mathrm{for}\qquad \frac{1}{p}+\frac{1}{q}\le \frac{1}{r}.
	\end{equation} 
\end{lemma}

\begin{proof}
	Using H\"older's inequality and then Jensen's inequality, as the integral of the mollifier is one and $1/x$ is a convex function we get that $\norm{\frac{1}{w^\eps}}\leq\norm{(\frac1w)^\eps}\leq\norm{\frac{1}{w}}$ and so 
	\begin{equation}
	\left\|\frac{w^\eps -w }{w^\eps}\right\|_{L^r}\le  \left\|w^\eps -w \right\|_{L^q}\left\|\frac{1}{w^\eps}\right\|_{L^p}\le  \left\|w^\eps -w \right\|_{L^q}\left\|\frac{1}{w}\right\|_{L^p}  \le C.
	\end{equation}
	As long as $q<\infty$ we see that this, in fact, converges to zero. 
\end{proof}

We now obtain the following corollary adding this condition into Theorem \ref{thm:compressibleonsagerCGamma}. We note that when $p=q=3$ then we obtain the best result with the weakest integrability assumption in the Besov norms.

\begin{corollary}\label{fastdescent}
		Let $\rho$, $u$ be a solution of~\eqref{eq:compressibleEuler} in the sense of distributions. Assume 
		\begin{align}
		u&\in B_p^{\alpha,\infty}((0,T)\times\T^d),\hspace{0.3cm}\rho,\rho u\in B_q^{\beta,\infty}((0,T)\times\T^d),\\
		0 &\leq \underline{\rho} \leq \rho \leq \overline{\rho} \ \mbox{a.e. (t,x) in}\; (0,T)\times\Tbb^d, 
		\end{align}
		for some constants $\underline{\vr}$, $\overline{\vr}$ and $0\leq\alpha,\beta\le1$ such that,
		\begin{equation}
		\frac1p + \frac2q \le 1,\qquad \frac2p + \frac1q \le 1,\qquad p,q\ge 2,  \qquad 	\alpha+\gamma\beta>1\qquad \text{and}\qquad
		2\alpha+\beta>1.
		\end{equation}
		Define  $\mathcal{E}:=\{x: \rho\neq0\}$ and assume that 
		\begin{equation}
		\frac{1}{\rho}\in {L^q(\mathcal{E})}.
		\end{equation}
		Assume further that $p\in C^{1,(\gamma-1)}([\underline{\rho},\overline{\rho}])$, and, in addition 
		\begin{equation}
		p'(0)=0\quad \text{as soon as}\quad \underline{\rho}=0.
		\end{equation}
			Then the energy is locally conserved, i.e.~\eqref{eq:energyidentity} holds
		in the sense of distributions on $(0,T)\times\Tbb^d$.
\end{corollary}  

\begin{proof}
	For the integral over $\mathcal{B}_{\eps^\beta}$, in the proof of Theorem \ref{thm:compressibleonsagerCGamma}, we see that as $\rho\in L^\infty$ and $\eps^\beta \ge \rho^\eps$ then
	\begin{align}
	\left|\int_0^T\int_{\mathcal{B}_{\eps^\beta}}\phi(\rho^\eps-\rho)\right.&\hspace{-0.14cm}\left.(u^\eps-u)(\rho^\eps)^{\gamma-1}\frac{\nabla \rho^\eps}{\rho^\eps}\d x\d t\right|\\\le &
	C\|\phi\|_{C^0}\eps^{\beta-1+\alpha}\|\rho\|_{B_q^{\beta,\infty}}\|u\|_{B_p^{\alpha,\infty}}\|(\rho^\eps)^{\gamma-1}\|_{L^\infty(\mathcal{B}_{\eps^\beta})}\left\|\frac{\rho^\eps-\rho}{\rho^\eps}\right\|_{L^q(\mathcal{B}_{\eps^\beta})}\\
	\le & C\|\phi\|_{C^0}\eps^{\gamma\beta-1+\alpha}\|\rho\|_{B_q^{\beta,\infty}}\|u\|_{B_p^{\alpha,\infty}}\left\|\frac{1}{\rho}\right\|_{L^q(\mathcal{E})}
	\end{align}
	and so we are done using lemma \ref{lemma:lprecipriocal} for the final step.
\end{proof}

\begin{remark}
	Even though we have written $\frac{1}{\rho}\in {L^q(\mathcal{E})}$ we   can fix some $\delta>0$ and only need this condition on some $\mathcal{B}_{\delta}$, as for $\eps^1>\eps^2$, then $\mathcal{B}_{\eps^2}\subset \mathcal{B}_{\eps^1}$ and so when $\eps^\beta <\delta$, then $\mathcal{B}_{\eps^\beta}\subset \mathcal{B}_{\delta}$.    
\end{remark}


One can see  that the condition 
$\frac{1}{\rho}\in {L^q(\mathcal{B}_\delta)}$ is quite a strong assumption and requires a quick approach of the function to the null set. Above we used conventional bounds to obtain a general integral result but do not consider the local structure of the function.
We notice that a point-wise estimate $\rho \le C \rho^\eps$ would allow to control the $L^q$ norm of $\frac{\rho^\eps-\rho}{\rho^\eps}$ and  though convexity of $\rho$ would do we will now show a nice link between this and quasi-nearly subharmonic functions which is defined in \cite{QNS}. 


%
 
 
 \begin{definition}
 	Let $X\subset\mathbb{R}^d$ be a set and  $u:X\to [0,+\infty)$ be Borel
 	measurable. Then $u$ is quasi-nearly subharmonic on $X$, $u\in QNS(X)$, if there is a constant
 	$\eps_0= \eps_0(u),$ 
 	$0<\eps_0<1$, such that for
 	each open set $O\subset X$, $O\neq X$, for each $x\in O$ and each $r$, $0<r\leq \eps_0\delta^{O}(x)$,
 	one has $u\in L^1(B_r(x))$ and
 	\begin{equation}\label{quasinearlysub}
 	u(x)\leq\frac{C}{|B_r(x)|}\int_{B_r(x)}u(y)\d y\quad\text{for some constant}\quad C\geq 1,
 	\end{equation}
 	where $C$ is independent of $r$, $|B_r(x)|=\omega_dr^d$ is the volume of the ball and 
 	\begin{equation}
 	\delta^{O}(x)=\dist(x, O^c)\quad\text{for the complement } O^c \text{ of } O \text{ in } X. 
 	\end{equation}
 \end{definition}

 \begin{lemma}\label{subharmonic}
 	Let $u:X\to[0, +\infty)$ be a Borel measurable function. Then $u$ is quasi-nearly subharmonic if and only if for every $O\subset\subset X$ there exist $M, \eps_0$ such that for any $0<\eps<\eps_0$
 	\begin{equation}
 	u(x)\leq M u^{\eps}(x)\quad\text{for any}\quad x\in O.
 	\end{equation}
 \end{lemma}
 \begin{proof}
 	Let $u:X\to[0, +\infty)$ be a quasi-nearly subharmonic function. Then for any $\eps<\dist(O, \partial X)$, $u^{\eps}$ is a well-defined smooth function on $O$. 
 	Suppose that $O\subset\subset X$ is a precompact set. Then 
 	$\delta_0=\dist(O, \partial X)$ is a positive number and for $\eps<\delta_0$,
 	\begin{equation}
 	O\subset\left\{x:\;\; \dist(x, \partial X)>\eps\right\}
 	\end{equation}
 	and $u^\eps$ is well-defined on $O$. 
 	We prove that there exist $M$ and $\eps_0$ such that 
 	\begin{equation}
 	u(x)\leq M u^{\eps}(x)\quad \text{for any}\quad x\in O,\quad 0<\eps<\eps_0.
 	\end{equation}
 	Indeed, we have 
 	\begin{equation}
 	u^{\eps}(x)=\frac{1}{\eps^d}\int_{X} \eta\left(\frac{x-y}{\eps}\right)u(y)\,\d y.
 	\end{equation}
 	Note that $y\in X$ for $x\in O$ and $|x-y|<\eps$.
 	Since $u\geq 0$ and recalling that from definition of $\eta$, we know that $\eta =1$ for $|x|<1/3$, we have
 	\begin{equation}
 	\begin{split}
 	u^\eps(x)&\geq\frac{1}{\eps^d}\int_{\left\{\left|\frac{x-y}{\eps}\right|\leq\frac{1}{3}\right\}} \eta\left(\frac{x-y}{\eps}\right)u(y)\,\d y\\
 	&=\frac{1}{\eps^d}\int_{\left\{\left|\frac{x-y}{\eps}\right|\leq\frac{1}{3}\right\}} u(y)\,\d y\\
 	&
 	=\frac{\omega_d}{3^d\left|B_{\frac{\eps}{3}}\left(x\right)\right|}\int_{B_{\frac{\eps}{3}}\left(x\right)} u(y)\,\d y\\
 	&
 	\geq\frac{\omega_d u(x)}{3^dC}
 	\end{split}
 	\end{equation}
 	for sufficiently small $\eps$.
 	Therefore, we obtain 
 	\begin{equation}
 	u(x)\leq\frac{3^dCu^{\eps}(x)}{\omega_d} \quad \text{for sufficiently small}\quad \eps\leq\eps_0\delta^{O}(x).
 	\end{equation}
 	On the other hand, if $u(x)\leq M u^{\eps}(x)$, then we have
 	\begin{equation}\label{eq:mollifierupperbound}
 	\begin{split}
 	u(x)&\leq \frac{M}{\eps^d}\int_{X} \eta\left(\frac{x-y}{\eps}\right)u(y)\,\d y\\
 	&
 	=\frac{M\omega_d}{\omega_d\eps^d}\int_{|x-y|\leq\eps} \eta\left(\frac{x-y}{\eps}\right)u(y)\,\d y\\
 	&\leq\frac{M\omega_d}{|B_\eps(x)|}\int_{B_{\eps}(x)}u(y)\,\d y.
 	\end{split}
 	\end{equation}  
 	Hence we deduce 
 	\begin{equation}
 	u(x)\leq\frac{C}{|B_\eps(x)|}\int_{B_{\eps}(x)}u(y)\,\d y.
 	\end{equation}
 	This completes the proof of the lemma.
 \end{proof}
 
 From this point-wise control showing that $\rho(x)\le M \rho^\eps(x)$ we obtain another corollary to our main result.
 
 \begin{corollary}\label{slowdescent}
 	Let $\rho$, $u$ be a solution of~\eqref{eq:compressibleEuler} in the sense of distributions. Assume 
 	\begin{align}
 	u&\in B_p^{\alpha,\infty}((0,T)\times\T^d),\hspace{0.3cm}\rho,\rho u\in B_q^{\beta,\infty}((0,T)\times\T^d),\\
 	0 &\leq \underline{\rho} \leq \rho \leq \overline{\rho} \ \mbox{a.e. in }\; (0,T)\times\Tbb^d, 
 	\end{align}
 	for some constants $\underline{\vr}$, $\overline{\vr}$ and $0\leq\alpha,\beta\le1$ such that,
 	\begin{equation}
 	\frac1p + \frac2q \le 1,\qquad \frac2p + \frac1q \le 1,\qquad p,q\ge 2,  \qquad 	\alpha+\gamma\beta>1\qquad \text{and}\qquad
 	2\alpha+\beta>1.
 	\end{equation}
 	Assume that $\rho\in QNS(\mathcal{B}_\delta)$ for some $\delta>0$ and $p\in\Ccal^{1,(\gamma-1)}([\underline{\rho},\overline{\rho}])$ with 
 	\begin{equation}
 	p'(0)=0\quad \text{as soon as}\quad \underline{\rho}=0.
 	\end{equation}
 	Then the energy is locally conserved, i.e.~\eqref{eq:energyidentity} holds
 	in the sense of distributions on $(0,T)\times\Tbb^d$.
 \end{corollary}  
 
 \begin{proof}
 	For the integral over $\mathcal{B}_{\eps^\beta}$, in the proof of Theorem \ref{thm:compressibleonsagerCGamma}, we see that 
 	\begin{equation}
 	\left\|\frac{\rho^\eps-\rho}{\rho^\eps}\right\|_{L^q(\mathcal{B}_{\eps^\beta})}\le \left\|\frac{\rho^\eps+ C\rho^\eps}{\rho^\eps}\right\|_{L^q(\mathcal{B}_{\eps^\beta})}\le C 
 	\quad\text{for}\quad \eps^\beta<\delta
 	\end{equation}
 	and so we are done.
 \end{proof}

\begin{remark}
	\begin{enumerate}
		\item We notice again that we only need that $\rho\in QNS(\mathcal{B}_\delta)$ so only in a neighbourhood of $\mathcal{B}_0:= \lim_{\eps \to 0} \mathcal{B}_{\eps^\beta}$.
		 \item  This condition already deals with the $\rho,\rho_\eps=0$ without splitting into cases and so using this condition the proof is simplified.
		 \item We note that this condition is weaker than local convexity of $\rho$ on $\mathcal{B}_{\delta}$ which would also give the same result. 
	\end{enumerate}
\end{remark}


\subsection{Counterexample for the $L^p$ case}

We indicate in this subsection why Lemma~\ref{mollificationlem} is no longer true when the $L^1$-norm is replaced with the $L^p$-norm for a $p>1$. This shows that the H\"older assumption of Corollary~\ref{HolderCor} cannot easily be relaxed.

  We can see $\rho^\eps(x)$ is like a weighted average of $\rho$ over the ball $B_\eps(x)$ and so heuristically we can see that
\begin{equation}
\frac{ \rho-\rho^\eps}{ \rho^\eps}\simeq\frac{\rho(x)-\frac{1}{|B_\eps|}\int_{B_\eps(x)}\rho(y)\d y}{\frac{1}{|B_\eps|}\int_{B_\eps(x)}\rho(y)\d y}
\end{equation}
(which is rigorous for $\eta_\eps = \frac{1}{|B_\eps|}\chi_{B_\eps(0)}(x)$) and assuming the right hand side is bounded and rearranging gives the condition \eqref{quasinearlysub}. We see that looking at a condition of the form 
\begin{equation}
\left\|\frac{\rho(\cdot)-\frac{1}{|B_\eps|}\int_{B_\eps(\cdot)}\rho(y)\d y}{\frac{1}{|B_\eps|}\int_{B_\eps(\cdot)}\rho(y)\d y}\right\|_{L^p}<C,
\end{equation}
in a sense a ``relatively weighted $L^p$ mean oscillation condition'', could potentially be the weakest condition to control \eqref{eq:mollifierratiocondition}.


We notice that for the $L^1$ norm we obtain perfect cancellation in the fraction when calculating \eqref{eq:L1cancleation}, as a mollifier acts like a local weighted average. However, when we perform the calculation in \eqref{eq:L1cancleation}, but in $L^p$, then instead we obtain
 \begin{equation}
 \sum_{j=1}^m C_d \frac{|B_\eps|}{\sum_{i=1}^{k_j} b_i |{D_i}|}\left\|\sum_{i=1}^{k_j} b_i \chi_{D_i}\right\|_{L^p(\Omega)}=\sum_{j=1}^m C_d \frac{|B_\eps|}{\sum_{i=1}^{k_j} b_i |{D_i}|}\left(\sum_{i=1}^{k_j} b_i^p |D_i|\right)^{1/p}
 \end{equation}
 and if we assume that the $b_i=1$ then we get $\sum_{j=1}^m C_d |B_\eps|\left(\sum_{i=1}^{k_j}  |D_i|\right)^{1/p-1}$. As $1/p-1<0$ then for certain functions this term could blow up. 
 
 In fact if one chooses a function made of separated spikes where the supports get smaller and smaller then we can show this blow up. We will formulate a simple  counterexample so that it is in one dimension, non-continuous and non-negative though more regular counter examples can be constructed in higher dimensions, that are for instance, even smooth and strictly positive. 

Firstly,  note that if we show that $\left\|\frac{f}{f^\eps}\right\|_{L^p}$ blows up as $\eps \to 0 $ then $\left\|\frac{f}{f^\eps}-\frac{f^\eps }{f^\eps }\right\|_{L^p}$ will also blow up. We can take $x\in \T$ and define  
our counter example 
\begin{equation}
f(x):=\sum^\infty_{i=1} \chi_{\left[\frac{1}{i},\frac{1}{i}+\frac{1}{2^{i}}\right]}(x).
\end{equation} 
It is easy to see that $f\in B_p^{\alpha, \infty}(\T)$ for $p>1$ and any $0<\alpha<1-\frac{1}{p}$ by regularizing and using Lemma 2.49 from \cite{bahouri2011fourier}.
So that we  have the sum of  separated spikes so they are further than $\frac{1}{i^2}$ apart yet have supports of size $\frac{1}{2^i}$. Let $\eps = \frac{1}{2i^2}$ and see that as $f$ is non-negative we can bound the sum below by just the $i$th spike and see that as mollification only acts locally, so the value on the denominator is only dependent on the  $i$th spike, thus we obtain
 \begin{equation}\label{eq:fractioonofmollificatio}
 \left\|\frac{f}{f^\eps}\right\|_{L^p(\T)}
 \ge\left\|\frac{1}{f^\eps}\right\|_{L^p\left(\frac{1}{i},\frac{1}{i}+\frac{1}{2^i}\right)} 
 =\left\|\left(\left(\chi_{\left[\frac{1}{i},\frac{1}{i}+\frac{1}{2^{i}}\right]}\right)^{\frac{1}{2i^2}}\right)^{-1}\right\|_{L^p\left(\frac{1}{i},\frac{1}{i}+\frac{1}{2^i}\right)}.
 \end{equation}
 We can then the bound mollification of $\chi_{\left[\frac{1}{i},\frac{1}{i}+\frac{1}{2^{i}}\right]}$ in a similar method to \eqref{eq:mollifierupperbound} but in one dimension and so we can bound \eqref{eq:fractioonofmollificatio} below by 
 \begin{equation}
 \left\|\frac{f}{f^\eps}\right\|_{L^p(\T)}
 \ge C\frac{2^i}{2i^2} \|1\|_{L^p\left(\frac{1}{i},\frac{1}{i}+\frac{1}{2^i}\right)}
=C\frac{2^i}{2i^2}2^{-i/p}
=C\frac{2^{i(1-1/p)}}{{2i^2} }.
 \end{equation}
As $f$ is the sum of  infinitely many   spikes  there will exist an appropriate spike for any  $\eps_i$ and thus 
we can send $i\to \infty$ and, as $1-1/p>0$, so $C\frac{2^{i(1-1/p)}}{{2i^2}} \to \infty$ which implies that $\left\|\frac{f}{f^\eps}\right\|_{L^p(\T)}\to \infty$.

\section{Energy Conservation on Domains with boundary}\label{section:with boundary}

We have derived the local energy conservation equations on $(0,T)\times \T^d$ and so for an $\phi\in C^\infty_c((0,T)\times \T^d)$ we have that
\begin{equation}\label{eq:localenergyequation}
\int_0^T\int_{\T^d}\partial_t\phi\cdot\left(\frac{1}{2}\rho|u|^2+P(\rho)\right)+\nabla \phi\cdot\left[\left(\frac{1}{2}\rho|u|^2+p(\rho)+P(\rho)\right)u\right]\d t \d x=0.
\end{equation}
The local energy equation is derived by taking momentum balance equations and testing with $(\phi u^\eps)^\eps$ and using that mollification is symmetric to regularise the equation. For  the continuity equation we just use $\phi ^\eps$ to test the equation and again move the mollification onto the equation.  Once this is done all the calculations are done locally on $\supp(\phi)$. 

When studying the isentropic Euler equations on a bounded domain with Lipschitz boundary $\Omega$ we have  
\begin{equation}\label{eq:compressibleEulerwithboundary}
\begin{aligned}
\partial_t(\rho u)+\diverg(\rho u\otimes u)+\nabla p(\rho)&=0,\qquad \mathrm{in} \quad[0,T]\times \Omega\\
\partial_t\rho+\diverg(\rho u)&=0,\qquad \mathrm{in}\quad [0,T]\times \Omega\\
u\cdot n&=0, \qquad \mathrm{on}\quad [0,T]\times \partial\Omega
\end{aligned}
\end{equation}
where $n$ denotes the outward normal vector field for $\partial\Omega$. For any $\phi\in C^\infty_c((0,T)\times \Omega)$ we can find an $\eps_0>0$ such that for all $0<\eps<\eps_0$ then both $\phi^\eps, (\phi u^\eps)^\eps \in C^\infty_c((0,T)\times \Omega)$ and so can apply the same method as above to obtain a local energy equation on $(0,T)\times \Omega$ of the form \eqref{eq:localenergyequation}. Here we are assuming the same conditions on $u,\rho$ and $p$ as in the previous theorems and in the corollaries in sections \ref{section:divergencevmeasure} and \ref{section:Holdercontinuity}, yet making the appropriate changes so that $u$ and $\rho$ are defined on the domain $(0,T)\times \Omega$ rather than $(0,T)\times \T^d$.

The following theorem and its proof follow ideas from~\cite{BaTiWi2018}:

\begin{theorem}
	Let $\rho$, $u$ be a solution of~\eqref{eq:compressibleEulerwithboundary} in the sense of distributions. Assume that $\rho,u$ and $p$ satisfy the conditions necessary to derive the local energy equality \eqref{eq:localenergyequation}. Assume further that $\rho\in L^\infty((0,T)\times \partial \Omega)$,  $\partial\Omega$ is $C^2$ and $u\cdot n$ is continuous at the boundary then we have energy conservation on $\Omega$, that is, for $\Theta(t)\in C^\infty_c(0,T)$
	\begin{equation}\label{eq:globalinspace}
	\int^T_0 \int_{\Omega}\partial_t\Theta(t)\cdot\left(\frac{1}{2}\rho|u|^2+P(\rho)\right)\d t \d x=0
	\end{equation}
	and further if $u,\rho$ are weakly continuous in time then 
	\begin{equation}\label{eq:globalinspaceandtime}
	\int_{\Omega} \frac{1}{2}\rho|u|^2(t_1,x) +P(\rho)(t_1,x)\d x=\int_{\Omega} \frac{1}{2}\rho|u|^2(t_2,x) +P(\rho)(t_2,x)\d x, 
	\end{equation}
	for any $t_1,t_2\in [0,T]$.
\end{theorem}
  
  \begin{proof}
  	For any $\phi\in C^\infty_c((0,T)\times \Omega)$ we can find an $\eps_0>0$ such that for all $0<\eps<\eps_0$ then both $\phi^\eps, (\phi u^\eps)^\eps \in C^\infty_c((0,T)\times \Omega)$ and so assuming sufficient regularity of $\rho, u$ and $p$ we obtain  
  	\begin{equation}\label{eq:localenergyequationwithboundary}
  	\int_0^T\int_{\Omega}\partial_t\phi\cdot\left(\frac{1}{2}\rho|u|^2+P(\rho)\right)+\nabla \phi\cdot\left[\left(\frac{1}{2}\rho|u|^2+p(\rho)+P(\rho)\right)u\right]\d t \d x=0.
  	\end{equation}
  	Let $\chi :\R^+ \to \R$ be a non-negative, smooth function such that
  	\begin{equation}
  	\chi(s):=\begin{cases}
  	0 & \quad \mathrm{if}\quad s<1\\
  	1 & \quad \mathrm{if}\quad s>2
  	\end{cases}
  	\end{equation}
  	and define for $x\in \overline \Omega$ the function $d_{\partial \Omega}(x)$ as the euclidean distance from $x$ to the closest point on the boundary. We can then define for any $\delta>0$ the composition $\chi\left(\frac{d_{\partial \Omega}(x)}{\delta}\right)$ and see that as $\delta\to 0$ so does  $\chi\left(\frac{d_{\partial \Omega}(x)}{\delta}\right) \to \mathbb{I}_{\Omega}$. Further, let $\Theta(t)\in C^\infty_c(0,T)$. 
  
  	We can for any $\delta>0$ let $\phi(x,t)=\chi\left(\frac{d_{\partial \Omega}(x)}{\delta}\right) \Theta(t)$ in \eqref{eq:localenergyequationwithboundary} and we obtain 
  	\begin{multline}
  	\int_{\Omega}\chi\left(\frac{d_{\partial \Omega}(x)}{\delta}\right)\int_0^T\partial_t\Theta(t)\cdot\left(\frac{1}{2}\rho|u|^2+P(\rho)\right)\d x\d t\\ +\int_0^T\Theta(t)\int_{\Omega}\nabla \chi\left(\frac{d_{\partial \Omega}(x)}{\delta}\right) \cdot\left[\left(\frac{1}{2}\rho|u|^2+p(\rho)+P(\rho)\right)u\right]\d t \d x=0
  	\end{multline}
  	and by the chain rule we see that   $\nabla \chi\left(\frac{d_{\partial \Omega}(x)}{\delta}\right)=\frac{1}{\delta}\chi'\left(\frac{d_{\partial \Omega}(x)}{\delta}\right)\nabla d_{\partial\Omega}(x)$ and so
  	\begin{multline}\label{eq:globalisinginspace}
  	0=\int_{\Omega}\chi\left(\frac{d_{\partial \Omega}(x)}{\delta}\right)\int_0^T\partial_t\Theta(t)\cdot\left(\frac{1}{2}\rho|u|^2+P(\rho)\right)\d x\d t\\ +\int_0^T\Theta(t)\int_{\Omega}\frac{1}{\delta}\chi'\left(\frac{d_{\partial \Omega}(x)}{\delta}\right)\nabla d_{\partial\Omega}(x) \cdot\left[\left(\frac{1}{2}\rho|u|^2+p(\rho)+P(\rho)\right)u\right]\d t \d x.
  	\end{multline}
  	As $\chi\left(\frac{d_{\partial \Omega}(x)}{\delta}\right) \to \mathbb{I}_{\Omega}$ strongly so the first integral on the R.H.S. of \eqref{eq:globalisinginspace} will converge to 
  	\begin{equation}
  	\int_0^T \int_{\Omega}\partial_t\Theta(t)\cdot\left(\frac{1}{2}\rho|u|^2+P(\rho)\right)\d x\d t
  	\end{equation}
  	as we wanted. All that is left is to show that the other term on the R.H.S. of \eqref{eq:globalisinginspace} vanishes in the limit.
  	
  	As $\partial \Omega$ is $C^2$ we can use \cite{gilbarg2015elliptic}, specifically Lemma 14.16, to see that there exists an $a>0$ such that $d_{\partial\Omega}(x)\in C^2(\Gamma_a)$ where $\Gamma_a:=\{x\in \overline\Omega: d_{\partial\Omega}(x)<a\}$. Further, in a similar argument to \cite{bardos2014non} Section 7,  when $x\in \Omega$ is sufficiently close to $\partial \Omega$ then there exists a unique point $\hat x\in \partial\Omega $ such that $x=\hat x + n(\hat x)d_{\partial \Omega}(x) $
  	where $n(\hat x)$ is the unit outward normal to the boundary at $x$. 
  	We see that we can bound the modulus  for the second term on the R.H.S.  of \eqref{eq:globalisinginspace} by
  	\begin{multline}\label{eq:lebequediffbouund}
  	\left\|\chi'\left(\frac{d_{\partial \Omega}}{\delta}\right)\right\|_{L^\infty}\int_0^T\Theta(t)\frac{1}{\delta}\int_{\Gamma_{2\delta}}|\nabla d_{\partial\Omega}(x) \cdot u|\left|\left(\frac{1}{2}\rho|u|^2+p(\rho)+P(\rho)\right)\right|\d t \d x\\
  	\le 	C\int_0^T\Theta(t)\frac{1}{\delta}\int_{\Gamma_{2\delta}}|\nabla d_{\partial\Omega}(x) \cdot u|\d t \d x
  	\end{multline}
  	as we know that $	\left\|\chi'\left(\frac{d_{\partial \Omega}}{\delta}\right)\right\|_{L^\infty}\le C$ and by our assumptions $\left\|\left(\frac{1}{2}\rho|u|^2+p(\rho)+P(\rho)\right)\right\|_{L^\infty} \le C$ as well.
  For  $2\delta<a$ we know that $d_{\partial\Omega}\in C^2$ and furthermore as $\nabla d_{\partial\Omega}\in C^1$, in the region $\Gamma_{2\delta}$,  $|\nabla d_{\partial\Omega}(x) \cdot u| \to C|n(\hat x)\cdot u(\hat x)|$ as  long as $u(x)\to u(\hat x)$ as $x\to \hat x$ and for this the assumption that $u\cdot n$ is continuous at the boundary will suffice.  Thus as $\partial \Omega$ is at least Lipschitz so    $|\Gamma_{2\delta}|\le C \delta |\partial \Omega|$ and so we can apply Lebesgue Differentiation theorem to  \eqref{eq:lebequediffbouund} and see that as $\delta \to 0$ so 
  \begin{equation}
  C\int_0^T\Theta(t)\frac{1}{\delta}\int_{\Gamma_{2\delta}}|\nabla d_{\partial\Omega}(x) \cdot u|\d t \d x\to C\int_0^T\Theta(t)\int_{\partial\Omega}|n(\hat x)\cdot u(\hat x)|\d t \d \hat x=0
  \end{equation}
  as $n(\hat x)\cdot u(\hat x)=0$ and so we have shown \eqref{eq:globalinspace}. 
  
  We now want to show \eqref{eq:globalinspaceandtime} with the extra assumptions of weak continuity in time of both $u$ and $\rho$. To do this we define  
    the sequence of functions  $\Theta_{\nu}: [0,T] \to \mathbb{R}$ which are non-negative, smooth  where for any point $t_1,t_2\in [0,T]$ where $t_1<t_2$ then
    	\begin{equation}
    	\Theta_{\nu}(\tau):=\begin{cases}
    	0 & \quad \mathrm{if}\quad \tau<t_1+\nu \quad \mathrm{or}\quad \tau> t_2-\nu\\
    	1 & \quad \mathrm{if}\quad \tau>t_1+2\nu \quad \mathrm{or}\quad \tau< t_2-2\nu
    	\end{cases}
    	\end{equation}
    	and see similarly that as $\nu \to 0$ so does $\Theta_{\nu}(t) \to \mathbb{I}_{[t_1,t_2]}$. We see that $\Theta_{\nu}\in C^\infty_c(0,T)$ for every $\nu>0$ and so substituting this function into \eqref{eq:globalisinginspace} we obtain 
    	\begin{equation}
    	\int^T_0 \int_{\Omega}\partial_t\Theta_{\nu}(t)\cdot\left(\frac{1}{2}\rho|u|^2+P(\rho)\right)\d t \d x=0
    	\end{equation} 
    	for every $\nu$. From our choice of $\Theta_{\nu}$ we see that
    	\begin{align}
    		\int^T_0 \int_{\T^d}\partial_t\Theta_{\nu}(t)\cdot\left(\frac{1}{2}\rho|u|^2+P(\rho)\right)\d t \d x&=\int^{t_1+2\nu}_{t_1} \partial_t\Theta_{\nu}(t)\cdot\int_{\Omega}\left(\frac{1}{2}\rho|u|^2+P(\rho)\right)\d t \d x\\
    		&+\int^{t_2}_{t_2-2\nu}\partial_t\Theta_{\nu}(t)\cdot \int_{\Omega}\left(\frac{1}{2}\rho|u|^2+P(\rho)\right)\d t \d x.
    	\end{align}
    	We know that $\int^{t_1+2\nu}_{t_1} \partial_t\Theta_{\nu}(t)\d t =1$ and $\int^{t_2}_{t_2-2\nu}\partial_t\Theta_{\nu}(t)\d t =-1$ by the fundamental theorem of calculus and as $\nu\to 0$ these terms approximate the identity at $t_1$ and $t_2$ and thus these terms converge to 
    	\begin{equation}
    		\int_{\Omega} \frac{1}{2}\rho|u|^2(t_1,x) +P(\rho)(t_1,x)\d x \quad\mathrm{and}\quad -\int_{\Omega} \frac{1}{2}\rho|u|^2(t_2,x) +P(\rho)(t_2,x)\d x
    	\end{equation}
    	respectively, assuming
    	weak continuity of $\rho$ and $u$ in time. Thus we are done.  
  \end{proof}

\subsection*{Acknowledgements}

We would like to thank the universities of Hannover and Warsaw for creating great environments for research. T.D acknowledges the support from the National Science Centre (Poland), 2012/05/E/ST1/02218. The research funding for J.S. is graciously given by German research foundation (DFG) grant no. WI 4840/1-1.

\end{document}